\documentclass[12pt]{ait}
\usepackage{graphicx}
\usepackage{cite}

\begin{document}  %%%!!!

\year{201}
\title{ОБЩИЕ ДИВЕРГЕНТНЫЕ УСЛОВИЯ УСТОЙЧИВОСТИ ДИНАМИЧЕСКИХ СИСТЕМ}%

\thanks{Результаты раздела 3 получены при поддержке гранта Российского фонда фундаментальных исследований (№ 17-08-01266).}

\authors{И.Б.~ФУРТАТ, д-р~техн.~наук (cainenash@mail.ru)\\
(ИПМаш РАН, Университет ИТМО, Санкт-Петербург)}
%П.А.~ГУЩИН, канд~техн.~наук (guschin.p@mail.ru)\\
%(РГУ нефти и газа (НИУ) имени И.М. Губкина, Москва)}

\maketitle

\begin{abstract}
Предложены новые необходимые и достаточные условия устойчивости динамических систем с использованием свойств потока и дивергенции вектора фазовой скорости. Данные условия обобщают известные результаты В.П. Жукова и А. Рантцера (A. Rantzer). Установлена связь между методами Ляпунова и предложенными методами. Применение разработанных результатов для исследования устойчивости линейных систем позволило свести задачу к вопросу разрешимости матричных неравенств. Получены методы синтеза законов управления для линейных и нелинейных систем. Приведены примеры, иллюстрирующие применимость предложенного метода и существующих.

\end{abstract}

\textit{Ключевые слова:} динамическая система, устойчивость, поток векторного поля, дивергенция, формула Остоградского-Гаусса, управление.

\section{Введение}

Динамическими моделями описывается множество процессов в окружающем макро- и микромире. Одним из важных вопросов эволюции таких систем является исследование сходимости решений данных моделей. Однако найти явное решение дифференциального уравнения не всегда представляется возможным, а численные решения могут значительно отличаться от истинного \cite{Afanasiev03}.

Хорошо известно, что методы функций Ляпунова позволяют установить устойчивость решений дифференциальных уравнений, не решая их. Впервые это показано А.М. Ляпуновым в конце 19 века в его докторской диссертации (позже опубликованной в \cite{Lyapunov50}) с применением к задачам астрономии и движения жидкости. Последующее развитие метода Ляпунова для исследования устойчивости и неустойчивости различных динамических систем, а также приложения полученных результатов в авиации, технике, механике и т.д. можно найти в следующих классических трудах \cite{Chetaev55,Liotov62,Malkin66,Zubov84,Rumiancev87}.
В зависимости от решаемых задач функция Ляпунова также интерпретируется  как потенциальная функция (potential function) \cite{Yuan14}, функция энергии (energy function) \cite{Bikdash00} или функция хранения (storage function) \cite{Willems72}. Однако основное ограничение в использовании аппарата функций Ляпунова состоит в поиске данных функций. Поэтому в настоящей статье предлагается рассмотреть альтернативный метод исследования устойчивости динамических систем на базе свойств потока и дивергенции вектора фазовой скорости.

Несмотря на то, что в западной литературе первенство в исследовании устойчивости с помощью дивергенции вектора фазовой скорости отдается А. Рантцеру \cite{Rantzer00,Rantzer01}, в отечественной литературе данная идея впервые была опубликована примерно на 20 лет раньше В.П. Жуковым в \cite{Jukov78}. В \cite{Jukov78} исследуется неустойчивость решения нелинейного дифференциального уравнения с помощью дивергенции векторного поля. Здесь же, как и в \cite{Rantzer00,Rantzer01}, для доказательства основного результата использовалась теорема Лиувилля \cite{Arnold74}. Затем в течение примерно 30 лет по исследованию неустойчивости различного вида динамических систем В.Н. Жуковым опубликован цикл работ, с частью которых можно свободно ознакомится на сайте журнала ''Автоматики и телемеханики''. Отметим только некоторые его статьи, в которых результаты близки к результатам в зарубежной литературе. В \cite{Jukov79} приведено необходимое условие устойчивости нелинейных систем в виде неположительности дивергенции векторного поля фазовой скорости. В \cite{Jukov90} впервые для исследования неустойчивости нелинейных динамических систем вводится вспомогательная скалярная функция. Позже в \cite{Rantzer00} А. Рантцер рассмотрел подобную функцию к данной скалярной функции, которую назвал функцией плотности (density function). В \cite{Jukov99} впервые получены условия устойчивости для систем второго порядка. Затем А. Рантцер в работах \cite{Rantzer00,Rantzer01} обсуждает сходимость почти всех решений нелинейных динамических систем произвольного порядка и рассматривает вопросы синтеза закона управления. Подход А. Рантцера отличается от подхода В.П. Жукова тем, что для исследования устойчивости он использует функцию плотности, которая подобна обратной вспомогательной функции В.П. Жукова, за исключением их свойств в точке равновесия динамической системы. В настоящее время дивергентный подход А. Рантцера \cite{Rantzer00,Rantzer01} распространен на различного рода системы \cite{Monzon03,Loizou08,Castaneda15,Karabacak18}.
В статье \cite{Furtat13} предложен совершенно другой способ исследования устойчивости динамических систем с использованием свойств потока вектора фазовой скорости через замкнутую выпуклую поверхность и установлена связь дивергентного метода со вторым методом Ляпунова.

Однако результатам \cite{Jukov79,Jukov99,Rantzer01,Furtat13} присущи следующие особенности:

1) необходимое условие \cite{Jukov79} достаточно усиленное;

2) метод \cite{Jukov99} обоснован только для систем второго порядка. Причем существует ряд систем второго порядка, для которых результаты \cite{Jukov99} не выполняются;

3) достаточные условия в \cite{Rantzer01} гарантируют сходимость только части решений системы в зависимости от начальных условий. Так, если условия \cite{Rantzer01} выполнены, то при одних начальных условиях решения сходятся к точке равновесия, а при других -- нет. Причем метод \cite{Rantzer01} не позволяет определить область притяжения. Более того, существует ряд неустойчивых систем во всей области определения для которых условия \cite{Rantzer01} также выполняются. Как следствие, условие устойчивости в \cite{Rantzer01} для линейных систем не всегда выполнено. В результате, алгоритмы управления, предложенные в \cite{Rantzer01}, не всегда гарантируют устойчивость замкнутой системы. Более того, в \cite{Rantzer01} для линейных систем условия работоспособности статического регулятора сведены к разрешимости нелинейного матричного неравенства, которое может быть выполнено и для неустойчивой замкнутой системы;

4) условия устойчивости в \cite{Furtat13} требует существования преобразования координат, которое приводит исходную систему к диагональному виду. Для нелинейных систем поиск такого преобразования является трудно решаемой задачей.

Настоящая статья посвящена разработке общего принципа исследования устойчивости динамических систем с использованием свойств потока и дивергенции вектора фазовой скорости, который позволит преодолеть особенности 1--4. Дополнительно, с помощью полученных результатов будут решены следующие задачи:

а) будут получены новые необходимые условия устойчивости, которые гарантируют точнее результаты, чем условия в \cite{Jukov79,Jukov99,Rantzer01}. Более того, будет показано, что достаточные условия сходимости почти всех решений в \cite{Rantzer01} на самом деле являются частным случаем полученных в статье необходимых условий устойчивости;

б) будет установлена связь метода Ляпунова с предложенными методами исследования устойчивости на базе потока и дивергенции вектора фазовой скорости;

в) будет показано, что скалярные функции, предложенные в \cite{Jukov90,Jukov99,Rantzer01}, не являются произвольными. Свойства данных функций зависят от поверхности интегрирования;

г) будут получены новые достаточные условия устойчивости. Для линейных систем данные условия будут сведены к разрешимости нелинейных и линейных матричных неравенств.

Статья сопровождается примерами с иллюстрацией особенностей 1--3 и соответствующими результатами использования нового подхода.

\section{Основной результат}
\label{Sec2}

Рассмотрим динамическую систему

\begin{equation}
\label{eq1}
\begin{array}{l} 
\dot{x}=f(x),
\end{array}
\end{equation} 
где $x=[x_1, ..., x_n]^{\rm T}$ -- вектор состояния, $f=[f_1,...,f_n]^{\rm T}: D \to \mathbb R^{n}$ -- непрерывно-дифференцируемая функция, определенная в области $D \subset \mathbb R^{n}$. Множество $D$ содержит начало координат и $f(0)=0$.
Ради простоты положим, что область притяжения $D_A$ точки $x=0$ совпадает с областью $D$. Однако все полученные результаты будут справедливы, если $D_A \subset D$ или $D_A=\mathbb R^{n}$. Обозначим $\bar{D}$ -- граница области $D$.

В статье будем использовать следующие обозначения: $grad \{S(x)\}=\Big[ \frac{\partial S}{\partial x_1}, ..., \frac{\partial S}{\partial x_n} \Big]^{\rm T}$ -- градиент функции $S(x)$, $div\{f(x)\}=\frac{\partial f_1}{\partial x_1}+...+\frac{\partial f_n}{\partial x_n}$ -- дивергенция векторного поля $f(x)$, $|\cdot|$ -- евклидова норма соответствующего вектора, $trace(A)$ -- след матрицы $A$. 
%Выводы в статье будут опираться на формулу Остоградского-Гаусса $F=\oint_{S} n f dS=\int_{V}div\{f\} dV$, где $F$ -- поток векторного поля $f$ через поверхность $S$, которая ограничивает объем $V$ в $\mathbb R^n$, $n^{\rm T}$ -- единичный вектор нормали к поверхности $S$. В литературе данную формулу можно также найти в форме divergence theorem (теорема о дивергенции) и Gauss theorem (теорема Гаусса). 
Под устойчивостью будем понимать устойчивость нулевого положения равновесия системы по Ляпунову  \cite{Khalil09}.

Сформулируем необходимое условие устойчивости \eqref{eq1}.

\begin{theorem}
\label{Th1}
Пусть $x=0$ -- асимптотически устойчивая точка равновесия системы \eqref{eq1}. Тогда существует положительно определенная непрерывно-дифференцируемая функция $S(x)$ такая, что $S(x) \to \infty$ при $x \to \bar{D}$, $|grad\{S(x)\}| \neq 0$ для любых $x \in D \setminus \{0\}$, $S(0)=0$ и выполнено одно из следующих условий:

1) $div\{|grad\{S(x)\}|f(x)\} < 0$ для любых $x \in D \setminus \{0\}$ и $div\{|grad\{S(x)\}|f(x)\} \big|_{x=0}=0$;

2) $div\{|grad\{S^{-1}(x)\}|f(x)\} > 0$ для любых $x \in D \setminus \{0\}$ и функция $div\{|grad\{S^{-1}(x)\}|f(x)\}$ интегрируема в области $\{x \in D: S^{-1}(x) \geq C\} \subset D$, где $C>0$.

\end{theorem}

%\textit{Замечание 1}. Проверки знаков интегралов можно избежать, если установить знакоопределенность подинтегральных функций. То есть, если $div\{|grad\{S(x)\}|f(x)\} < 0$, то $\int_{V} div\{|grad\{S(x)\}|f(x)\} dV < 0$, а также, если $div\{|grad\{S^{-1}(x)\}|f(x)\} > 0$, то и $\int_{V_{inv}} div\{|grad\{S^{-1}(x)\}|f(x)\} dV_{inv} > 0$. Также заметим, что если $D=\mathbb R^n$, то функция $S(x)$ является радиально неограниченной.

%%%%%%%%%%%%%%%%%%%%%%%%%%%%%%%%%%%%%%%%%%%%%%%%

\begin{proof}
Согласно \cite[теорема 4.17]{Khalil09}, если $x=0$ -- асимптотически устойчивая точка равновесия системы \eqref{eq1}, то существует непрерывно-дифференцируемая положительно-определенная функция $S(x)$ такая, что $S(x) \to \infty$ при $x \to \bar{D}$ и $grad\{S(x)\}^{\rm T}f(x) \leq -W(x)$ для любых $x \in D \setminus \{0\}$, где $W(x)$ -- положительно-определенная функция. Значит $grad\{S(x)\}^{\rm T}f(x) < 0$ для любых $x \in D \setminus \{0\}$ и $grad\{S(x)\}^{\rm T}f(x) \Big|_{x=0} = 0$. Заметим, что если $D=\mathbb R^n$, то функция $S(x)$ является радиально неограниченной.

Рассмотрим доказательство двух случаев по отдельности.

1) Обозначим $V=\{x \in D: S(x) \leq C\} \subset D$ и $\Gamma=\{x \in D: S(x)=C\}$. Если $grad\{S(x)\}^{\rm T}f(x) < 0$ при $x \in D \setminus \{0\}$, то и $\frac{1}{|grad\{S(x)\}|} grad\{S(x)\}^{\rm T} |grad\{S(x)\}|f(x) < 0$. Значит будет справедливо следующее выражение
\begin{equation*}
\label{eq2}
\begin{array}{l}
F_1=\oint_{\Gamma} \frac{1}{|grad\{S(x)\}|} grad\{S(x)\}^{\rm T} |grad\{S(x)\}|f(x) d\Gamma < 0 ~~~ \mbox{при} ~~~ x \in D \setminus \{0\}.
\end{array}
\end{equation*} 
Здесь $F_1$ -- поток векторного поля $|grad\{S(x)\}|f(x)$ через поверхность $\Gamma$ с единичным вектором нормали $\frac{1}{|grad\{S(x)\}|} grad\{S(x)\}$. Воспользовавшись формулой Остоградского-Гаусса (в литературе ее также можно найти в виде divergence theorem (теорема о дивергенции) и Gauss theorem (теорема Гаусса)), получим
\begin{equation}
\label{eq3}
\begin{array}{l}
F_1=\int_{V}div\{|grad\{S(x)\}| f(x)\} dV < 0 ~~~ \mbox{при} ~~~ x \in D \setminus \{0\}.
\end{array}
\end{equation} 
Поскольку $C$ -- произвольная константа, то \eqref{eq3} будет выполнено для любого достаточно малого значения $C$. Обозначим $\bar{V}$ -- объем области $V$. Так как $\bar{V} \to 0$ при $C \to 0$, то из соотношений $div\{|grad\{S(x)\}| f(x)\}=\lim_{\bar{V} \to 0} F_1 / \bar{V}$, $F_1<0$ и $\bar{V}>0$ следует, что $div\{|grad\{S(x)\}| f(x)\}<0$ при $x \in D \setminus \{0\}$.  Из условия $grad\{S(x)\}^{\rm T}f(x) \big|_{x=0}=0$ следует, что $F_1 = 0$ при $x=0$. Тогда из \eqref{eq3} следует, что $div\{|grad\{S(x)\}|f(x)\} \big|_{x=0}=0$. Геометрическая иллюстрация случая 1 приведена на рис. \ref{Fig010}.

\hspace{-3cm}
\textbf{Рис. ~\ref{Fig010}. $\to$}

2) Пусть $V_{inv}=\{x \in D: S^{-1}(x) \geq C\} \subset D$ и $\Gamma_{inv}=\{x \in D: S^{-1}(x)=C\}$. Если $grad\{S(x)\}^{\rm T}f(x) < 0$ и $S(x)>0$ при $x \in D \setminus \{0\}$, то $grad\{S^{-1}(x)\}^{\rm T}f(x)=-S^{-2}(x)grad\{S(x)\}^{\rm T}f(x) > 0$. С другой стороны $grad\{S^{-1}(x)\}^{\rm T}f(x)=\frac{1}{|grad\{S^{-1}(x)\}|}grad\{S^{-1}(x)\}^{\rm T} |grad\{S^{-1}(x)\}| f(x)$. Значит будет выполнено следующее соотношение
\begin{equation*}
\label{eq4}
\begin{array}{l}
F_2=\oint_{\Gamma_{inv}}\frac{1}{|grad\{S^{-1}(x)\}|}grad\{S^{-1}(x)\}^{\rm T} |grad\{S^{-1}(x)\}| f(x) d \Gamma_{inv} > 0 ~~~ \mbox{при} ~~~ x \in D \setminus \{0\}.
\end{array}
\end{equation*}
Здесь $F_2$ -- поток векторного поля $|grad\{S^{-1}(x)\}|f(x)$ через поверхность $\Gamma_{inv}$ с единичным вектором нормали $\frac{1}{|grad\{S^{-1}(x)\}|} grad\{S^{-1}(x)\}$. Очевидно, что знаки $F_1$ и $F_2$ противоположны, поскольку векторы $grad \{S(x)\}$ и $grad \{S^{-1}(x)\}$ противоположно направлены.
Воспользовавшись формулой Остоградского-Гаусса, получим
\begin{equation}
\label{eq5}
\begin{array}{l}
F_2=\int_{V_{inv}}div\{|grad\{S^{-1}(x)\}| f(x)\} dV_{inv} > 0 ~~~ \mbox{при} ~~~ x \in D \setminus \{0\}.
\end{array}
\end{equation}
Так как константа $C$ произвольная, то \eqref{eq5} будет выполнено для любого достаточно большого значения $C$. Пусть $\bar{V}_{inv}$ -- объем области $V_{inv}$. Поскольку $\bar{V}_{inv} \to 0$ при $C \to \infty$, то из $div\{|grad\{S^{-1}(x)\}| f(x)\}=\lim_{\bar{V}_{inv} \to 0} F_2 / \bar{V}_{inv}$, $F_2>0$ и $\bar{V}_{inv}>0$ следует, что $div\{|grad\{S^{-1}(x)\}| f(x)\}>0$ при $x \in D \setminus \{0\}$.
С другой стороны $S^{-1}(x) \to \infty$ при $x \to 0$. Следовательно, из выражения \eqref{eq5} требуется интегрируемость (в несобственном смысле) функции $div\{|grad\{|S^{-1}(x)\}| f(x)\}$ в области $V_{inv}$. Геометрическая иллюстрация случая 2 приведена на рис. \ref{Fig011}. Теорема \ref{Th1} доказана.

\hspace{-3cm}
\textbf{Рис. ~\ref{Fig011}. $\to$}
\end{proof}

Таким образом в теореме \ref{Th1} отмечается, что если система \eqref{eq1} устойчива, то поток векторного поля через поверхность $\Gamma$ ($\Gamma_{inv}$) принимает отрицательное (положительное) значение. Обратное не верно. Поток может принимать отрицательное (положительное) значение не только в случае направленности векторного поля во внутрь поверхности $\Gamma$ ($\Gamma_{inv}$), но и в случае сходимости части компонент векторного поля к нулю и расходимости остальных компонент. Другим примером отрицательного (положительного) значения потока может быть случай, когда при одних начальных условиях решения сходятся, а при других -- расходятся. Соответствующие примеры 2 и 3 будут приведены ниже.
Физический смысл результатов теоремы \ref{Th1} может быть интерпретирован как движение вещества с плотностью $|grad\{S(x)\}|$ ($|grad\{S^{-1}(x)\}|$) под действием поля скорости $f(x)$ через поверхность $\Gamma$ ($\Gamma_{inv}$).

Условия теоремы \ref{Th1} зависят от вида функции $S(x)$, которая связана с поверхностью интегрирования. Сформулируем теорему, которая позволит ослабить данное требование.

\begin{theorem}
\label{Th2}
Пусть $x=0$ -- асимптотически устойчивая точка равновесия \eqref{eq1}. Тогда существует положительно определенная непрерывно-дифференцируемая функция $\rho(x)$, определенная для всех $x \in D$, такая что выполнено одно из условий:

1) $div\{\rho(x)f(x)\} < 0$ для любых $x \in D \setminus \{0\}$ и $div\{\rho(x)f(x)\} \big |_{x=0}=0$;

2) $div\{\rho^{-1}(x)f(x)\} > 0$ для любых $x \in D \setminus \{0\}$ и функция $div\{\rho^{-1}(x)f(x)\}$ интегрируема в области $\{x \in D: \rho^{-1}(x) \geq C\} \subset D$, где $C>0$.

\end{theorem}

%\textit{Замечание 2}. Для проверки условий 1 и 2 достаточно проверить знакоопределенность подинтегральных функций. То есть, если $div\{\rho(x)f(x)\} < 0$, то $\int_{V} div\{\rho(x)f(x)\} dV < 0$, а также, если $div\{\rho^{-1}(x)f(x)\} > 0$, то и $\int_{V_{inv}} div\{\rho^{-1}(x)f(x)\} dV_{inv} > 0$. Также заметим, что если $D=\mathbb R^n$, то функция $\rho(x)$ является радиально неограниченной.

Требование $div\{\rho^{-1}(x)f(x)\} > 0$ приведено в качестве достаточного условия сходимости почти всех решений в \cite{Rantzer01}. Очевидно, что данное условие \cite{Rantzer01} является частным требованием к выполнению необходимого условия 2 в теореме \ref{Th2}.

\begin{proof} Доказательство теоремы \ref{Th2} опирается на доказательство теоремы \ref{Th1}. Введем положительно определенную непрерывно-дифференцируемую скалярную функцию $\phi(x)$. Следуя доказательству теоремы \ref{Th1}, рассмотрим два случая.

1) Если $grad\{S(x)\}^{\rm T}f(x) < 0$ при $x \in D \setminus \{0\}$, то и $\phi(x) grad\{S(x)\}^{\rm T}f(x) < 0$. Следовательно, дальнейшее доказательство аналогично доказательству случая 1 в теореме \ref{Th1}, рассмотрев только поток векторного поля $\phi(x)|grad\{S(x)\}|f(x)$ через поверхность $\Gamma$. Обозначив $\rho(x)=\phi(x) |grad\{S(x)\}|$, получим условие 1 теоремы \ref{Th2}.

2) Если $grad\{S(x)\}^{\rm T}f(x) < 0$ при $x \in D \setminus \{0\}$, то и $\phi^{-1}(x) grad\{S(x)\}^{\rm T}f(x) < 0$. Значит дальнейшее доказательство аналогично доказательству случая 2 теоремы \ref{Th1}, но с учетом потока векторного поля $\phi^{-1}(x)|grad\{S^{-1}(x)\}|f(x)$ через поверхность $\Gamma_{inv}$. Обозначив $\rho^{-1}(x)=\phi^{-1}(x) |grad \{S^{-1}(x)\}|$, получим условие 2. Теорема \ref{Th2} доказана.

\end{proof}

Условия теоремы \ref{Th2} позволили не использовать функцию $S(x)$ для исследования устойчивости.  Тем не менее из доказательства теоремы \ref{Th2} следует, что свойства новой функции $\rho(x)$ зависят от свойств поверхности интегрирования, а значит и от свойств $S(x)$. Таким образом, выбор функции $\rho(x)$ не является произвольным, как это отмечалось в \cite{Jukov99, Rantzer01}.  

Рассмотрим применение теоремы \ref{Th2} на следующих примерах и сравним результаты с достаточными условиями, приведенными в \cite{Jukov99,Rantzer01}.

\textit{Пример 1.}
Задана система

\begin{equation}
\label{eq4}
\begin{array}{l} 
\dot{x}_1=x_2,
\\
\dot{x}_2=-c x_1-x_1^2 x_2-x_2^3,
\end{array}
\end{equation} 
где $c$ -- постоянный коэффициент, который может быть равен $1$ или $-1$. При $c=1$ система \eqref{eq4} устойчива, при $c=-1$ -- неустойчива. Фазовые портреты такой системы приведены на рис.~\ref{Fig3}. Исследуем сначала на устойчивость \eqref{eq4} с помощью условий из \cite{Jukov79,Jukov99,Rantzer01}. Затем приведем результаты теоремы \ref{Th2}.

\hspace{-3cm}
\textbf{Рис. ~\ref{Fig3}. $\to$}

Согласно \cite{Jukov79}, необходимым условием устойчивости является выполнение неравенства $div \{f(x)\} \leq 0$ в окрестности точки равновесия. В результате $div\{f(x)\}=-x_1^2-3x_2^2 \leq 0$ для любых значений $c$.

Согласно достаточному условию \cite{Jukov99}, система \eqref{eq4} обладает $\omega$-инвариантным множеством $B_T$ (см. рис.~\ref{Fig3}). Также, в соотвествии с \cite{Jukov99} (стр. 37, первый абзац) функцию $\mu(x)$ можно выбрать в виде $\mu(x)=1$. Тогда $div \{\mu(x)f(x)\} = 0$ на $mes = 0$ и $div \{\mu(x)f(x)\} = -x_1^2-3x_2^2 \leq 0$. Таким образом условия \cite{Jukov99} выполнены для любых $c$, как и в предыдущем случае.

Достаточное условие \cite{Rantzer01} связано с проверкой неравенства $div\{\rho^{-1}(x)f(x)\} < 0$. Выберем $\rho(x)=|x|^{2\alpha}$, где $\alpha$ -- натуральное число. Тогда функция $div\{\rho^{-1}(x) f(x)\}$ не является положительной ни при $c=1$, ни при $c=-1$. Следовательно, вывод о сходимости решений к точке равновесия сделать нельзя.

Теперь исследуем \eqref{eq4} с помощью теоремы \ref{Th2}. Выбрав $\rho(x)=|x|^{2\alpha}$, получим $div\{\rho(x) f(x)\}=2\alpha |x|^{2\alpha-2}(x_1 x_2 - cx_1 x_2-x_1^2 x_2^2-x_2^4)-(x_1^2+3x_2^2) |x|^{2\alpha} < 0$ для любых $\alpha$, $c=1$ и $x \neq 0$, а также $div\{\rho(x) f(x)\}\big|_{x=0}=0$. При $c=-1$ функция $div\{\rho(x) f(x)\}$ не является знакоопределенной. Таким образом, первое условие теоремы \ref{Th2} выполнено только для $c=1$, что соответствует устойчивому положению равновесия.

Несмотря на то, что теорема \ref{Th2} дала точнее результат, чем условия \cite{Jukov79,Jukov99,Rantzer01}, тем не менее условия теоремы \ref{Th2} необходимые. Приведем по этому поводу еще два примера, а затем сформулируем достаточные условия устойчивости.

\textit{Пример 2.} Рассмотрим систему

\begin{equation}
\label{eq04}
\begin{array}{l} 
\dot{x}_1=-x_1,
\\
\dot{x}_2=b x_2-x_1^2 x_2,~~~b>0,
\end{array}
\end{equation} 
которая устойчива только по части переменных, т.е. $x_1 \to 0$ и $x_2 \to \infty$ при  $t \to \infty$, см. рис.~\ref{Ust_chast_per}. Как и в предыдущем примере выберем $\rho(x)=|x|^{2\alpha}$. Тогда $div\{\rho(x) f(x)\}= |x|^{2\alpha-2}[(-2\alpha+b-1)x_1^2+(2\alpha b+b-1)x_2^2-2\alpha x_1^2 x_2^2-x_1^2(x_1^2+x_2^2)] < 0$ при $b < \frac{1}{1+2\alpha}$ и $x \neq 0$, а также $div\{\rho(x) f(x)\}\big|_{x=0}=0$. Таким образом первое условие теоремы \ref{Th2} выполнено. Выбрав $\mu(x)=\rho(x)$ в \cite{Jukov99} получили бы, что система \eqref{eq04} устойчива при том же $b < \frac{1}{1+2\alpha}$. Однако в \cite{Jukov99} получены достаточные условия устойчивости, в отличие от теоремы \ref{Th2}. Достаточные условия \cite{Rantzer01} с $\rho^{-1}(x)=|x|^{-2\alpha}$ не дают ответа о сходимости решений поскольку функция $div\{\rho^{-1}(x) f(x)\}$ не является знакоопределенной.

\hspace{-3cm}
\textbf{Рис. ~\ref{Ust_chast_per}. $\to$}

\textit{Пример 3.} Рассмотрим систему

\begin{equation}
\label{eq004}
\begin{array}{l} 
\dot{x}_1=-x_1+x_1^2-x_2^2,
\\
\dot{x}_2=-x_2+2x_1 x_2,
\end{array}
\end{equation} 
которая имеет две точки равновесия $(0,0)$ и $(1,0)$. Выберем $\rho^{-1}(x)=|x|^{-2\alpha}$. Тогда $div\{\rho^{-1}(x) f(x)\}=|x|^{-2\alpha}[2\alpha -2+2x_1(2-\alpha)] > 0$ при $\alpha = 2$. Таким образом второе условие теоремы \ref{Th2} выполнено. Однако все траектории системы сходятся к точке $(0,0)$, за исключением тех, что начинаются на полуоси $x_1 \geq 1$ и $x_2 = 0$ (см. рис.~\ref{Ust_Rantz}). Данный вывод согласуется с необходимым условием \ref{Th2}. Однако такой же результат  получим с помощью достаточного условия \cite{Rantzer01}. Значит  условия \cite{Rantzer01} не гарантирует устойчивость системы \eqref{eq004} и не позволяет определить область притяжения.

\hspace{-3cm}
\textbf{Рис. ~\ref{Ust_Rantz}. $\to$}

Теперь сформулируем  достаточные условия устойчивости.

\begin{theorem}
\label{Th2a}
Пусть задана положительно определенная непрерывно-дифференцируемая функция $\rho(x)$ определенная в области $D$. Тогда точка $x=0$ устойчива (асимптотически устойчива), если выполнено одно из следующих условий:

1) $div\{\rho(x)f(x)\} \leq \rho(x) div\{f(x)\}$ ($div\{\rho(x)f(x)\} < \rho(x) div\{f(x)\}$) для любых $x \in D \setminus \{0\}$ и $div\{\rho(x)f(x)\}\big|_{x=0} =0$;

2) $div\{\rho^{-1}(x)f(x)\} \geq 0$ ($div\{\rho^{-1}(x)f(x)\} > 0$) и $div\{f(x)\} \leq 0$ для любых $x \in D \setminus \{0\}$ и $\lim_{|x| \to 0} \big[\rho^2(x) div\{\rho^{-1}(x)f(x)\}\big] =0$;

3) $div\{\rho(x)f(x)\} \leq \beta(x) \rho^2(x) div\{\rho^{-1}(x)f(x)\}$ ($div\{\rho(x)f(x)\} < \beta(x) \rho^2(x) div\{\rho^{-1}(x)f(x)\}$), где $\beta(x) > 1$ и $div\{f(x)\} \leq 0$ или только $\beta(x) = 1$ для любых $x \in D \setminus \{0\}$, а также $div\{\rho(x)f(x)\}\big|_{x=0} =0$ и $\lim_{|x| \to 0} \big[\rho(x) div\{\rho^{-1}(x)f(x)\}\big] =0$.

\end{theorem}

Достаточные условия 1 и 2 можно рассматривать как некоторую связь с соответствующими необходимыми условиями 1 и 2 в теореме \ref{Th2}. Условие 3 является комбинацией условий 1 и 2, которое в теореме \ref{Th7} позволит получить обобщенное для случаев 1 и 2 условие устойчивости линейных систем.

\begin{proof} Приведем доказательство устойчивости для каждого случая в отдельности. Доказательство асимптотической устойчивости аналогично.

1) Из соотношения $div\{\rho(x)f(x)\} = grad \{\rho(x)\}^{\rm T}f(x)+div\{f(x)\}\rho(x)$ следует, что если $div\{\rho(x)f(x)\} \leq div\{f(x)\}\rho(x)$, то и $grad \{\rho(x)\}f(x) \leq 0$ в области $D \setminus \{0\}$. По условию $\rho(0)=0$. Поэтому, если $div\{\rho(x)f(x)\}\big|_{x=0} =0$, то и $grad \{\rho(x)\}f(x) \big|_{x=0}= 0$. Значит, согласно теореме Ляпунова \cite{Khalil09}, система \eqref{eq1} устойчива.

2) Из выражения $div\{\rho^{-1}(x)f(x)\} = grad \{\rho^{-1}(x)\}^{\rm T}f(x)+div\{f(x)\}\rho^{-1}(x)$ следует, что $grad \{\rho(x)\}^{\rm T}f(x)=\rho(x) div\{f(x)\} -\rho^{2}(x)div\{\rho^{-1}(x)f(x)\}$. Если $div\{\rho^{-1}(x)f(x)\} \geq 0$ и $div\{f(x)\} \leq 0$, то $grad \{\rho(x)\}^{\rm T}f(x) \leq 0$ в области $D \setminus \{0\}$. Если $\lim_{|x| \to 0} \big[\rho^2(x) div\{\rho^{-1}(x)f(x)\}\big] =0$, то и $\lim_{|x| \to 0} \big[grad \{\rho(x)\}f(x) \big]= 0$. Значит система \eqref{eq1} устойчива.

3) Условие 3 состоит в объединении результатов условий 1 и 2. Суммируя $\beta(x) grad \{\rho(x)\}^{\rm T}f(x)=\beta(x) \rho(x) div\{f(x)\} -\beta(x) \rho^{2}(x)div\{\rho^{-1}(x)f(x)\}$ и $grad\{\rho(x)\}^{\rm T}f(x) = div\{\rho(x)f(x)\}-div\{f(x)\}\rho(x)$, получим $(1+\beta(x)) grad \{\rho(x)\}^{\rm T}f(x)=div\{\rho(x)f(x)\}-\beta(x) \rho^{2}(x)div\{\rho^{-1}(x)f(x)\}+(\beta(x)-1) \rho(x) div\{f(x)\}$. Если $div\{\rho(x)f(x)\} \leq \beta(x) \rho^2(x) div\{\rho^{-1}(x)f(x)\}$ при $\beta(x)=1$ или $\beta(x)>1$ и $div\{f(x)\} \leq 0$, то $grad \{\rho(x)\}^{\rm T}f(x) \leq 0$  в области $D \setminus \{0\}$. Если $div\{\rho(x)f(x)\}\big|_{x=0} =0$ и $\lim_{|x| \to 0} \big[\rho^2(x) div\{\rho^{-1}(x)f(x)\}\big] =0$, то и $\lim_{|x| \to 0} \big[grad \{\rho(x)\}f(x) \big]= 0$. Значит система \eqref{eq1} устойчива.
Теорема \ref{Th2a} доказана.

\end{proof}

В \cite{Jukov99} отмечалось, что дивергентный подход \cite{Jukov99} может быть применим только для систем второго порядка. Продемонстрируем предложенный алгоритм для систем третьего порядка, а затем сформулируем условие устойчивости для линейных систем произвольного порядка и сравним их с условиями \cite{Rantzer01}.

\textit{Пример 4.}
Рассмотрим систему

\begin{equation}
\label{eq6}
\begin{array}{l} 
\dot{x}_1=x_2-2x_1 x_3^2,
\\
\dot{x}_2=-x_1-2x_2x_3^2,
\\
\dot{x}_3=-2x_3^3.
\end{array}
\end{equation} 

Выберем $\rho(x)=|x|^{2\alpha}$, где $\alpha$ -- натуральное число. Тогда $div\{\rho(x)f(x)\} - \rho(x) div\{f(x)\}=-4 \alpha x_3^2|x|^{2\alpha}<0$ для любых $\alpha$ и $x_3 \neq 0$. В свою очередь $div\{f(x)\}=-10x_3^2 \leq 0$ и $div\{\rho^{-1}(x) f(x)\}=(4\alpha-10)x_3^2|x|^{-2\alpha}>0$ для любых $\alpha \geq 3$ и $x_3 \neq 0$. Пусть $\beta(x)=\beta \geq 1$. Тогда $div\{\rho(x)f(x)\}-\beta \rho^2(x) div\{\rho^{-1}(x)f(x)\} = -(4\alpha+10+4\beta \alpha -10 \beta)x_3^2|x|^{2\alpha} \leq 0$ при $\alpha > \frac{5(\beta -1)}{\beta+1}$ и $x_3 \neq 0$. Все три случая дали одинаковые результаты. Таким образом, система \eqref{eq6} асимптотически устойчива с любыми начальными условиями, когда $x_3(0) \neq 0$. Если начальные условия содержат $x_3(0) = 0$, то система \eqref{eq6} устойчива. Фазовые траектории системы \eqref{eq6} изображены на рис.~\ref{Fig5}, где цикл получен для начального условия с $x_3=0$, спирали -- при $x_3 \neq 0$.

\hspace{-3cm}
\textbf{Рис. ~\ref{Fig5}. $\to$}

\section{Устойчивость линейных систем}
\label{Sec4}

Во введении отмечалось, что для линейных систем $\dot{x}=Ax$ дивергентный подход был рассмотрен в \cite{Rantzer01}, где получены достаточные условия сходимости решений. Приведем данный результат в виде теоремы, как он был сформулирован в \cite{Rantzer01} и ниже приведем примеры, когда результат \cite{Rantzer01} не выполняется. 

\begin{theorem}[\cite{Rantzer01}]
\label{Th4}
Если положительно определенная матрица $P$ удовлетворяет условию
\begin{equation}
\label{eq7}
\begin{array}{l} 
A^{\rm T}P+PA<\alpha^{-1} trace(A) P
\end{array}
\end{equation} 
для некоторого $\alpha>0$, тогда линейная система $\dot{x}=Ax$ устойчива.
\end{theorem}

Неравенство \eqref{eq4} выполняется не только для гурвицевых, но и для негурвицевых матриц $A$. Например, для неустойчивой матрицы $A=\left[\begin{array}{cccc} {0} & {1} \\ {1} & {1} \end{array}\right]$ и $\alpha=0.2$ решение \eqref{eq7} можно определить в виде $P=\left[\begin{array}{cccc} {0.6} & {0.3} \\ {0.3} & {0.9} \end{array}\right]$. Также, для данной матрицы $A$ и $\alpha \leq 0.2$ нетрудно найти положительно определенные матрицы $P$, удовлетворяющие неравенству \eqref{eq7}. Более того, при $\alpha \leq 1$ и гурвицевой матрице $A=\left[\begin{array}{cccc} {0} & {1} \\ {-1} & {-1} \end{array}\right]$ можно указать $P$, которые не являются положительно определенными матрицами. Например, при $\alpha = 1$ имеем $P=\left[\begin{array}{cccc} {-1.5} & {-0.75} \\ {-0.75} & {-1.5} \end{array}\right]$.

Отметим, что неравенство \eqref{eq7} можно получить с помощью условия 2 теоремы \ref{Th2}, выбрав $\rho^{-1}(x)=(x^{\rm T}Px)^{-\alpha}$, но при $\alpha>1$ для обеспечения интегрируемости $div\{\rho^{-1}(x)f(x)\}$. Более того, с помощью условия 1 теоремы \ref{Th2} можно получить другое условие

\begin{equation}
\label{eq07}
\begin{array}{l} 
A^{\rm T}P+PA+\alpha^{-1}trace(A)P < 0,
\end{array}
\end{equation}
выбрав $\rho(x)=(x^{\rm T}Px)^{\alpha}$ и $\alpha>0$. Однако матричные неравенства \eqref{eq7} и \eqref{eq07} являются необходимыми условиями устойчивости.

Теперь, опираясь на результаты теоремы \ref{Th2a} сформулируем достаточное условие устойчивости для линейных систем.

\begin{theorem}
\label{Th7}
Линейная система $\dot{x}=Ax$ устойчива, если
\begin{equation}
\label{eq08}
\begin{array}{l} 
A^{\rm T}P+PA-\kappa~ trace(A)P<0,
\end{array}
\end{equation} 
где $\kappa > 0$ и $trace(A) \leq 0$ или только $\kappa = 0$, $P$ -- положительно определенная симметричная матрица.
\end{theorem}

\begin{proof} Выберем $\rho(x)=(x^{\rm T}Px)^{\alpha}$, где $\alpha>0$.
Рассмотрим случай 1 теоремы \ref{Th2a}. Тогда $div\{\rho(x)f(x)\} - \rho(x) div\{f(x)\} = \alpha (x^{\rm T}Px)^{\alpha-1}x^{\rm T}[A^{\rm T}P+PA]x < 0$, если $A^{\rm T}P+PA<0$.

Теперь рассмотрим случай 2 теоремы \ref{Th2a}. Условие $div\{\rho^{-1}(x)f(x)\} = -\alpha (x^{\rm T}Px)^{-\alpha-1}x^{\rm T}[A^{\rm T}P+PA-\alpha^{-1} trace(A)P]x >0$ будет выполнено, если $A^{\rm T}P+PA-\alpha^{-1} trace(A)P < 0$ и $trace(A) \leq 0$ (следует из условия $div\{f(x)\} \leq 0$) при $\kappa=\frac{1}{\alpha} > 0$.

Воспользуемся случаем 3 теоремы \ref{Th2a}. Пусть $\beta(x)=\beta$ и $\kappa=\frac{\beta-1}{\alpha(1+\beta)}$. Условие $div\{\rho(x)f(x)\}-\beta \rho^2(x) div\{\rho^{-1}(x)f(x)\} = \alpha(1+\beta)(x^{\rm T}Px)^{\alpha-1}x^{\rm T}[A^{\rm T}P+PA-\frac{1}{\alpha} \frac{\beta-1}{\beta+1} trace(A)P]x < 0$ выполнено, если выполнено \eqref{eq08} при $\beta = 1$ ($\kappa = 0$) или при $\beta > 1$ ($\kappa > 0$) и $trace(A) \leq 0$. Теорема \ref{Th7} доказана.

Отметим, что при синтезе закона управления условие \eqref{eq08} является линейным матричным неравенством только при $\kappa=0$. Однако ЛМН также можно получить с использованием теоремы \ref{Th2a} и для $\kappa \neq 0$.
\end{proof}

\textit{Следствие 1.}
Если в теореме \ref{Th2a} случае 2 положить $\alpha=trace(A) \gamma$, $\gamma>0$, то для устойчивости $\dot{x}=Ax$ достаточно выполнение ЛМН $A^{\rm T}P+PA+\gamma P<0$.

\section{Синтез закона управления}
\label{Sec5}

Рассмотрим динамическую систему аффинную по управлению

\begin{equation}
\label{eq9}
\begin{array}{l} 
\dot{x}=f(x)+g(x)u(x),
\end{array}
\end{equation} 
где $u(x)$ -- сигнал управления, функции $f(x)$, $g(x)$ и $u(x)$ -- непрерывно-дифференцируемые в области $D$, причем $f(0)=0$ и $g(0)=0$. Следуя теореме \ref{Th2}, сформулируем следующий результат.

\begin{theorem}
\label{Th6}
Пусть задана положительно определенная непрерывно-дифференцируемая функция $\rho(x)$ при $x \in D$. Тогда система \eqref{eq9} с законом управления $u=u(x)$ устойчива (асимптотически устойчива), если выполнено одно из следующих условий:

1) $div\{\rho(x)(f(x)+g(x)u(x))\} \leq \rho(x) div\{f(x)+g(x)u(x)\}$ ($div\{\rho(x)(f(x)+g(x)u(x))\} < \rho(x) div\{f(x)+g(x)u(x)\}$) для любых $x \in D \setminus \{0\}$ и $div\{\rho(x)(f(x)+g(x)u(x))\}\big|_{x=0}=0$;

2) $div\{\rho^{-1}(x)(f(x)+g(x)u(x))\} \geq 0$ ($div\{\rho^{-1}(x)(f(x)+g(x)u(x))\} > 0$) для любых $x \in D \setminus \{0\}$ и $\lim_{|x| \to 0} \big[\rho^2(x) div\{\rho^{-1}(x)(f(x)+g(x)u(x))\}\big] =0$;

3) $div\{\rho(x)(f(x)+g(x)u(x))\} \leq \beta(x) \rho^2(x) div\{\rho^{-1}(x)(f(x)+g(x)u(x))\}$, $\beta \geq 1$ ($div\{\rho(x)(f(x)+g(x)u(x))\} < \beta(x) \rho^2(x) div\{\rho^{-1}(x)(f(x)+g(x)u(x))\}$), где $\beta(x) > 1$ и $div\{f(x)+g(x)u(x)\} \leq 0$ или только $\beta(x) = 1$ для любых $x \in D \setminus \{0\}$, а также $div\{\rho(x)(f(x)+g(x)u(x))\}\big|_{x=0} =0$ и $\lim_{|x| \to 0} \big[\rho(x) div\{\rho^{-1}(x)(f(x)+g(x)u(x))\}\big] =0$.

\end{theorem}

Доказательство теоремы \ref{Th6} непосредственно следует из доказательства теоремы \ref{Th2a}.

Отметим, что при синтезе закона управления с использованием функции Ляпунова $V(x)$ требуется выбрать $u$ так, чтобы было выполнено алгебраическое неравенство $grad\{V\}(f+gu)<0$. Согласно теореме \ref{Th6}, $u$ необходимо выбрать так, чтобы было выполнено дифференциальное неравенство, что дает новое условие поиска закона управления.

Применим теорему \ref{Th6} для стабилизации линейных систем.

\textit{Следствие 2.}
Рассмотрим систему $\dot{x}=Ax+Bu$ с управляемой парой $(A,B)$ и закон управления $u=Kx$. Пусть выполнено одно из следующих условий

\begin{equation}
\label{eq010}
\begin{array}{l}
(A+BK)^{\rm T}P+P(A+BK)-\kappa~ trace(A+BK)P<0,
\end{array}
\end{equation}
или
\begin{equation}
\label{eq10}
\begin{array}{l}
(A+BK)^{\rm T}P+P(A+BK)+\gamma P<0,
\end{array}
\end{equation} 
где $P$ -- положительно определенная симметричная матрица, $\kappa > 0$ и $trace(A+BK) \leq 0$ или только $\kappa = 0$, $\gamma>0$. Тогда замкнутая система устойчива.

Доказательство следствия 2 непосредственно следует из теоремы \ref{Th7} и следствия 1 в силу уравнения замкнутой системы $\dot{x}=(A+BK)x$.

Поскольку неравенства \eqref{eq010} (при $\kappa = 0$) и \eqref{eq10} являются линейными, то для расчета матрицы $K$ можно использовать любые существующие программные пакеты, решающие ЛМН. Если же воспользоваться теоремой \ref{Th4} \cite{Rantzer01}, то матричное неравенство \eqref{eq7} при наличии управления $u=Kx$ нельзя свести к ЛМН. Более того, как было показано ранее, неравенство \eqref{eq7} из \cite{Rantzer01} (заменив $A$ на $A+BK$) может иметь решение для неустойчивой матрицы $A+BK$, что не приведет к синтезу стабилизирующего регулятора. Например, для $A=\left[\begin{array}{cccc} {0} & {1} \\ {1} & {1} \end{array}\right]$, $B=\left[\begin{array}{cccc} {0} \\ {1} \end{array}\right]$, $K=[-0.7082 -2.2651]$ и $\alpha=1$, решение \eqref{eq7} (заменив $A$ на $A+BK$) определено в виде $P=\left[\begin{array}{cccc} {0.7712} & {0.3508} \\ {0.3508} & {1.122} \end{array}\right]$, а собственные числа матрицы $A+BK$ приблизительно равны $0.2$ и $-1.5$.

\section{Заключение}
\label{Sec6}

Предложены новые методы исследования устойчивости линейных и нелинейных динамических систем с использованием свойств потока и дивергенции вектора фазовой скорости, которые обобщают результаты \cite{Jukov99,Rantzer01,Furtat13}. Для исследования устойчивости требуется существование определенного вида поверхности интегрирования или вспомогательной скалярной функции, однако свойства которой зависят от свойств поверхности интегрирования. Сформулированы необходимые и достаточные условия устойчивости. Применение полученных результатов для линейных динамических систем позволяет свести задачу исследования устойчивости к вопросу разрешимости новых матричных неравенств, а также ЛМН. Результаты применены к синтезу законов управления для линейных и нелинейных динамических систем. Показано, что для выбора закона управления требуется разрешить дифференциальное неравенство относительно сигнала управления, в то время как при использовании аппарата функций Ляпунова требуется разрешить алгебраическое неравенство.

$ $

$ $

\AdditionalInformation{Фуртат И.Б.}{ИПМаш РАН, ведущий научный сотрудник, Университет ИТМО, профессор, Санкт-Петербург}{cainenash@mail.ru}

\newpage

Рис. 1. Геометрическая иллюстрация случая 1 теоремы \ref{Th1}.

Рис. 2. Геометрическая иллюстрация случая 2 теоремы \ref{Th1}.

Рис. 3. Фазовые траектории системы \eqref{eq4} при $c=1$ (слева) и $c=-1$ (справа).

Рис. 4. Фазовый портрет \eqref{eq04} при $b=0.1$.

Рис. 5. Фазовый портрет системы \eqref{eq004} с двумя точками равновесия.

Рис. 6. Фазовые траектории системы \eqref{eq6}.

\newpage
\pagestyle{empty}

\begin{figure}[h!]
\center{\includegraphics[width=0.9\linewidth]{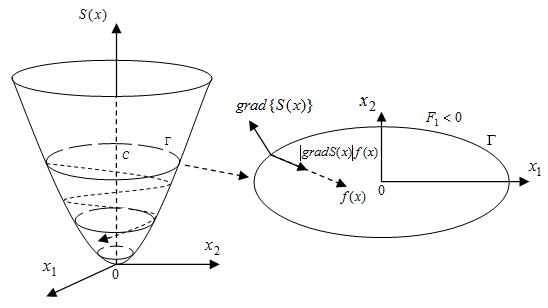}}
\caption{}
\label{Fig010}
\end{figure}

\newpage
\begin{figure}[h!]
\center{\includegraphics[width=0.7\linewidth]{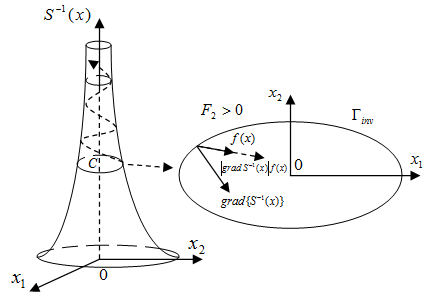}}
\caption{}
\label{Fig011}
\end{figure}

\newpage
\begin{figure}[h!]
\begin{minipage}[h]{0.47\linewidth}
\center{\includegraphics[width=1\linewidth]{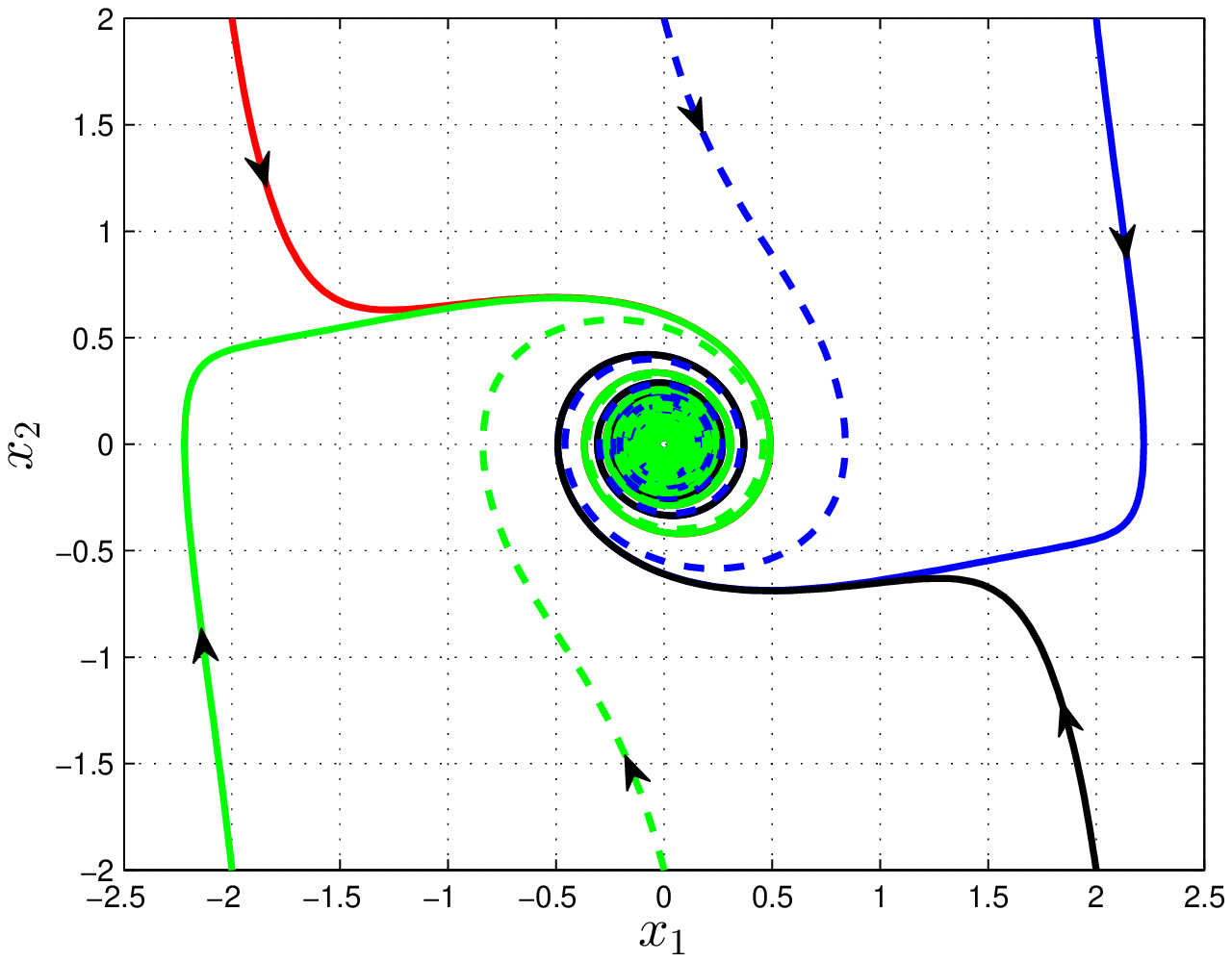}} \\
\end{minipage}
\hfill
\begin{minipage}[h]{0.47\linewidth}
\center{\includegraphics[width=1\linewidth]{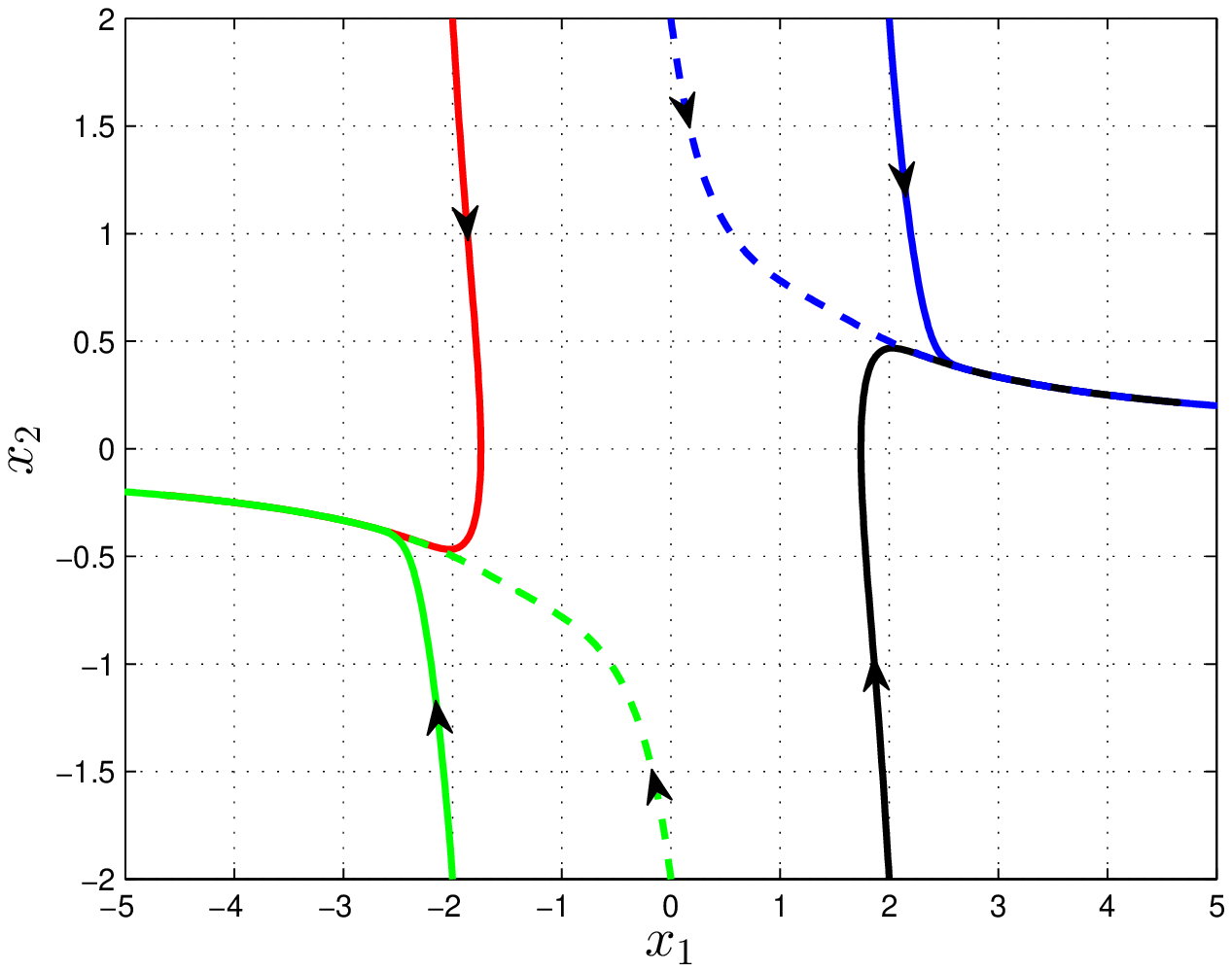}} \\
\end{minipage}
\caption{}
\label{Fig3}
\end{figure}

\newpage
\begin{figure}[h!]
\center{\includegraphics[width=0.7\linewidth]{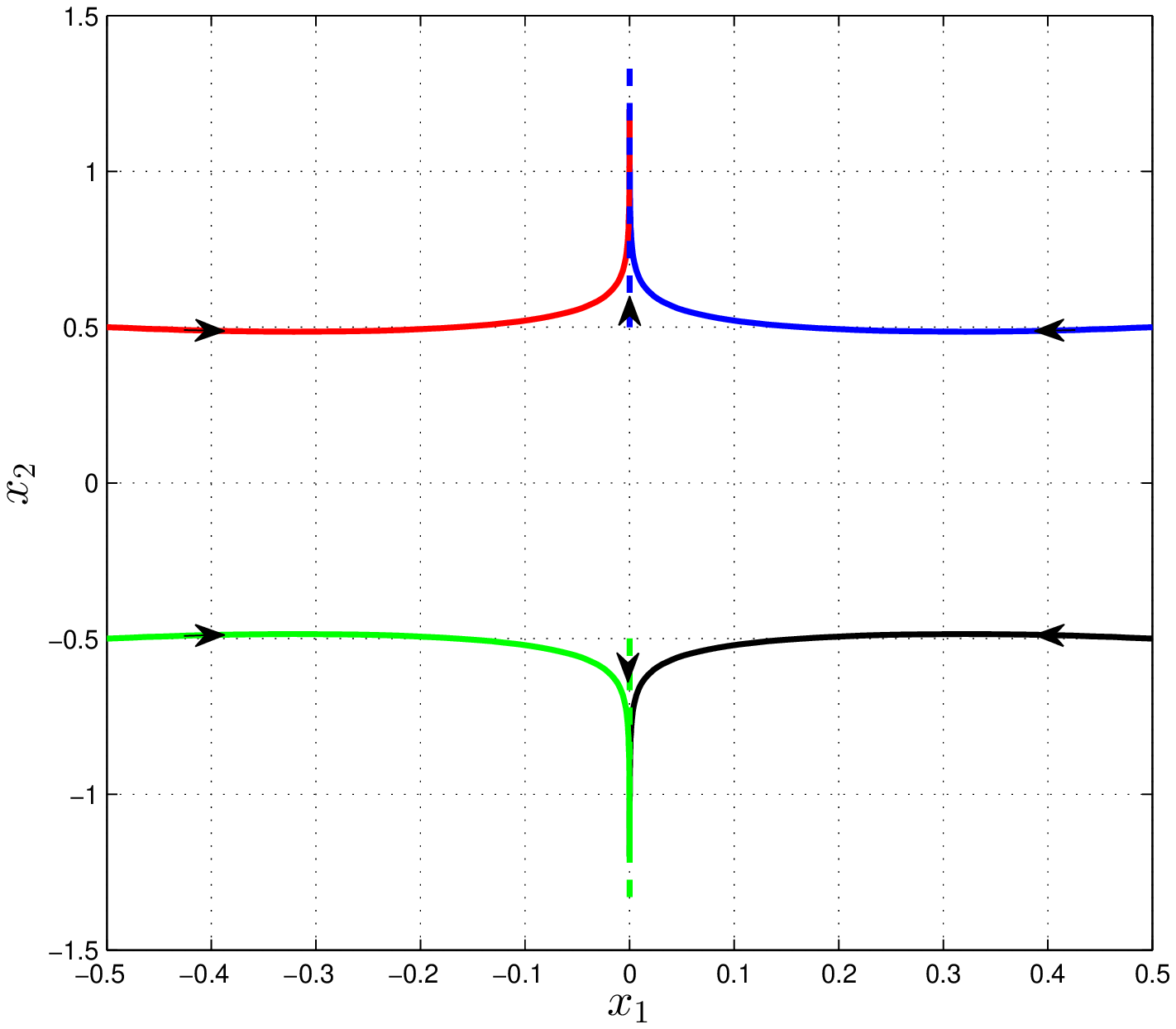}}
\caption{}
\label{Ust_chast_per}
\end{figure}

\newpage
\begin{figure}[h!]
\center{\includegraphics[width=0.7\linewidth]{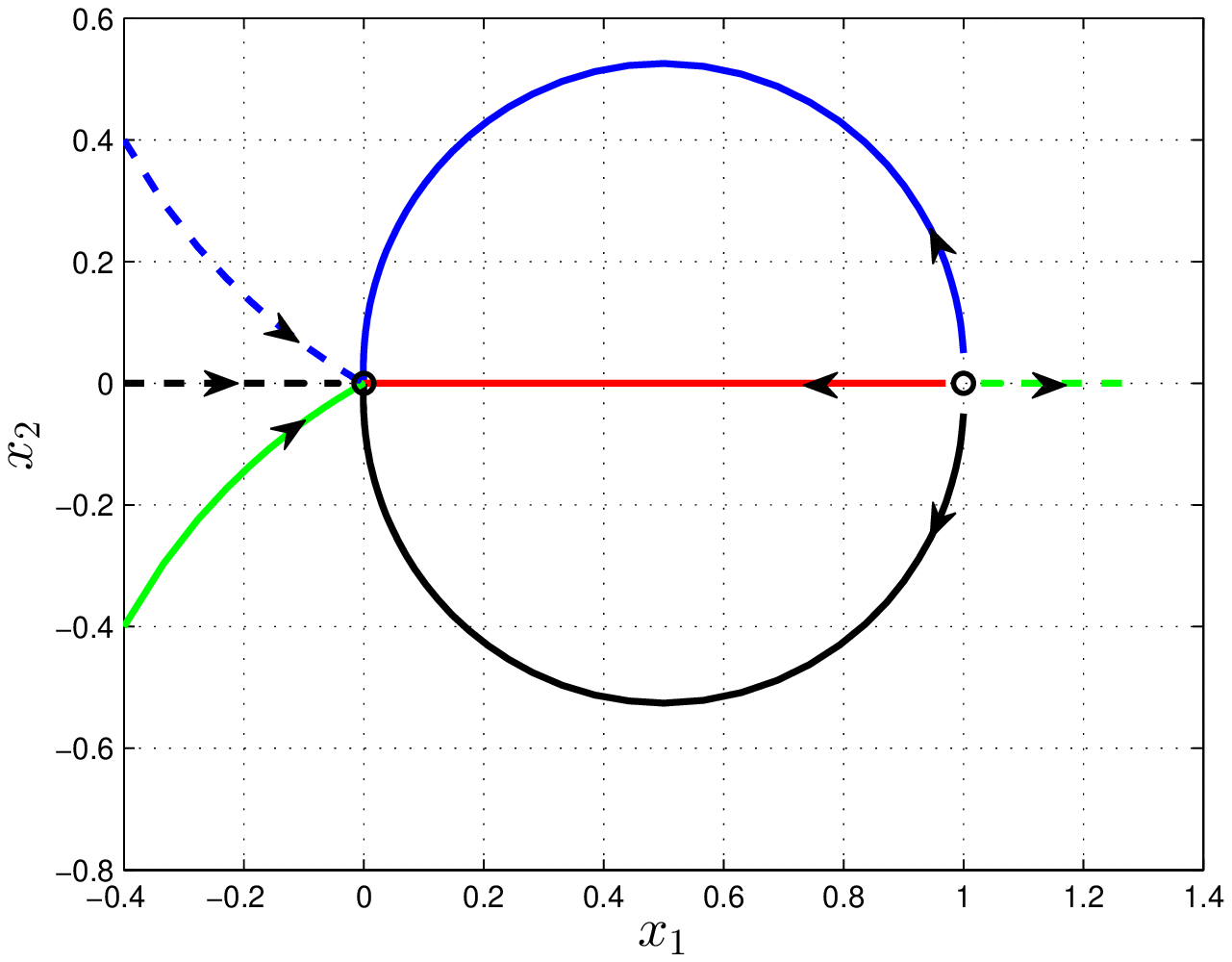}}
\caption{}
\label{Ust_Rantz}
\end{figure}

\newpage
\begin{figure}[h!]
\center{\includegraphics[width=0.7\linewidth]{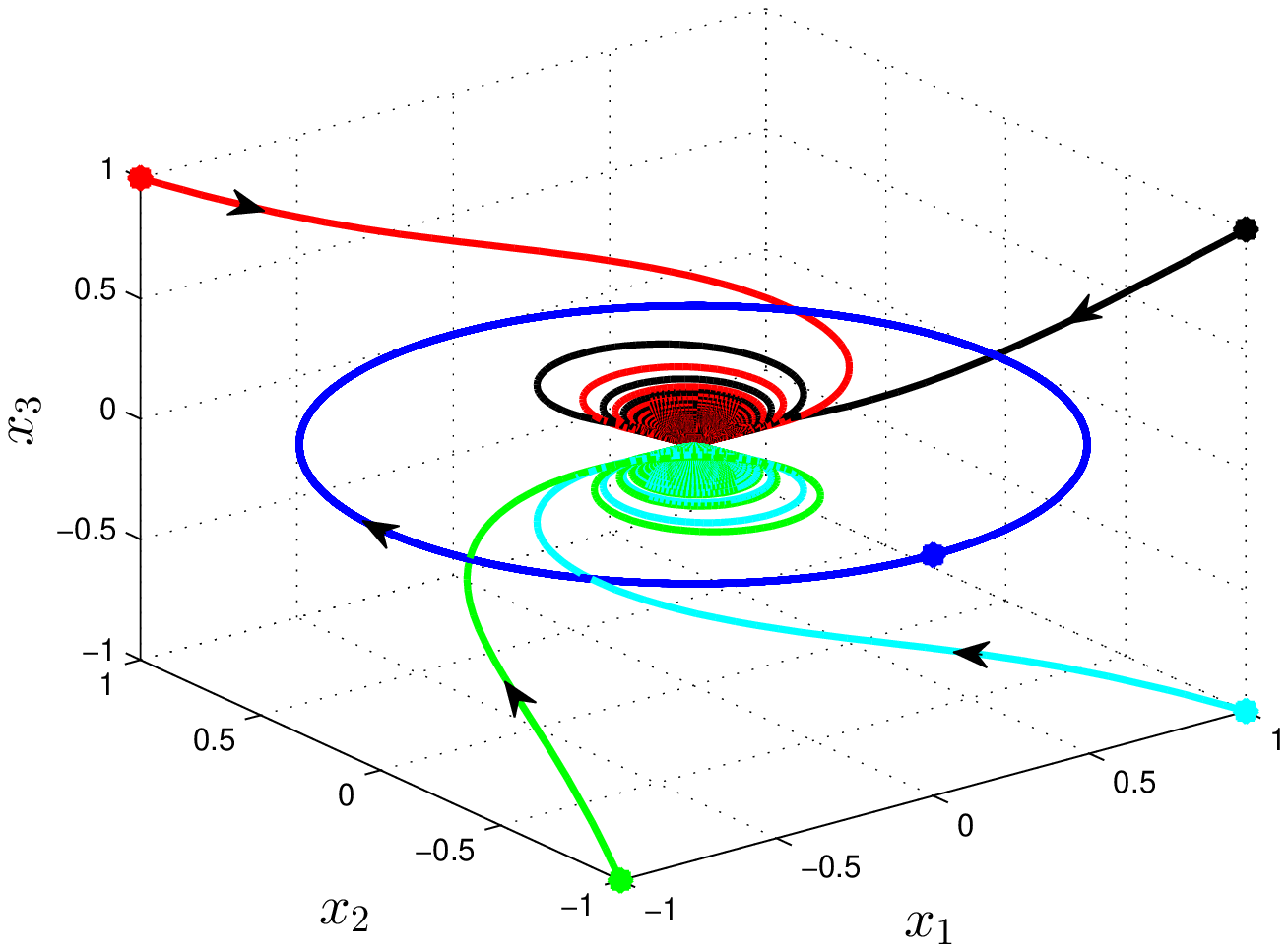}}
\caption{}
\label{Fig5}
\end{figure}

\newpage

\textit{Abstract}

New necessary and sufficient conditions are proposed for the stability investigation of dynamical systems using the flow and the divergence of the phase vector velocity. The obtained conditions generalize the well-known results of V.P. Zhukov and A. Rantzer. The relation of Lyapunov methods with the proposed methods is established. The application of the obtained results to study the stability of linear systems goes to the problem of matrix inequality solvability. The new control laws are  synthesized for linear and nonlinear systems. Examples illustrate the applicability of the proposed method and show the comparison results with some existing ones.

\textit{Key words:} dynamic system, stability, flow of a vector field, divergence, Gauss theorem, control.

\end{document}